\documentclass[reqno]{amsart}
\usepackage{inputenc,amsmath,amsthm,amssymb,amsfonts,mathrsfs,tikz,enumerate}
\usepackage{tikz-cd}
\usepackage[ruled,lined]{algorithm2e}

\newtheorem{thm}{Theorem}[section]
\newtheorem{lem}[thm]{Lemma}
\newtheorem{cor}[thm]{Corollary}
\newtheorem{prop}[thm]{Proposition}
\newtheorem{defn}[thm]{Definition}

\def \F {\mathbb{F}}
\def \cU {\mathcal{U}}
\def \cH {\mathcal{H}}
\def \tsr {\otimes}
\def \ker {\mathrm{ker}~}
\def \im {\mathrm{im}~}

\def \infZ {\mathbb{Z}_{+, \infty} }
\def \Z {\mathbb{Z}}
\def \PC {\mathcal{PC} }

\def \d {\delta }
\def \gdim {\mathrm{gdim}}

\def \Fx {\mathbb{F}[x]}

\def \fxc [#1][#2] {\Fx || #1 || \{ #2 \} }
\def \fc [#1][#2] {\F || #1 || \{ #2 \} }

\def \flt {\mathcal{F}}

\newcommand{\U}[1][\infty] {U_{ n , s , #1 }}

\title{Decomposition of Persistent Homology and Spectral Sequences}
\author{Peiqi Yang}
\address{Department of Mathematics, The George Washington University \\ Phillips Hall, 801 22nd St. NW, Room 739,
Washington, DC 20052, USA}
\email{pqyang@gwu.edu}

\author{Yingfeng Hu}
\address{Department of Mathematics, The George Washington University \\ Phillips Hall, 801 22nd St. NW, Room 739,
Washington, DC 20052, USA}
\email{yingfenghu@gwu.edu}

\author{Hao Wu}
\address{Department of Mathematics, The George Washington University \\ Phillips Hall, 801 22nd St. NW, Room 739,
Washington, DC 20052, USA}
\email{haowu@gwu.edu}

\date{}
\begin{document}

\begin{abstract}
We study the relation between the persistent homology and the spectral sequence of a filtered chain complex over a field. Our method is based on a decomposition of the persistent homology. We demonstrate that, under fairly general assumptions, these two algebraic structures capture the same information from the filtered chain complex.
\end{abstract}

\subjclass[2020]{Primary 55N31, 55T05}

\keywords{filtered chain complex, persistent homology, spectral sequence, exact couple}

\maketitle

\section{Introduction}

\subsection{Background}

Filtered chain complexes arise naturally in many applications of algebraic topology. For example, given a point cloud in a metric space, one can consider the abstract simplicial complex consisting of all the possible abstract simplexes with vertices in this cloud. This abstract simplicial complex gives rise a simplicial chain complex, which, by itself, has trivial homology. But, when equipped with the filtration induced by the diameters of simplexes, the filtered chain complex captures much topological information of the point cloud. 

There are two main approaches to extracting homological information from a filtered chain complex. The more classical approach is via spectral sequences defined by Jean Leray \cite{MR16665}, which are hard to compute in most cases. More recently, in the study of point clouds, most researchers opt to use the persistent homology because there are efficient algorithms to compute it \cite{ComTop,ComPH}. 

A natural question is how different or similar are the homological data obtained by these two approaches. Saugata Basu and Laxmi Parida \cite{SSPH} established explicit relations between the spectral sequences and the persistent homology for filtered chain complexes over a field. In fact, similar relations are implicit in the work of the third author \cite{EqKR} in the context of the Khovanov-Rozansky homology in knot theory. In this paper, we apply the arguments in \cite{EqKR} to more general filtered chain complexes over fields and, among other things, recover the results in \cite{SSPH}. Our method is based on decomposition theorems of persistent homology and spectral sequences, which allow us to avoid complicated exact sequence arguments and provide a direct and intuitive explanation why the two approaches yield exactly the same information from a filtered chain complex over a field under fairly general assumptions. 

\subsection{Statement of results} In the current paper, we fix a base field $\F$ and consider filtered chain complexes of $\F$-spaces.

\begin{defn}\label{def-flt-complex}
A filtered chain complex $(C,d, \flt)$ over $\F$ consists of 
\begin{enumerate}
	\item a $\mathbb{Z}$-graded $\F$-space $C=\bigoplus_{n\in \mathbb{Z}} C_n$,
	\item a differential map $d:C\rightarrow C$ satisfying $d^2=0$ and $d(C_n) \subset C_{n-1}$,
	\item an ascending $\mathbb{Z}$-filtration $\flt$ on each $C_n$ satisfying $d(\flt^p C_n)\subset \flt^p C_{n-1}$.
\end{enumerate}
We write $\flt^p C = \bigoplus_{n\in \mathbb{Z}} \flt^pC_n$. We say that 
\begin{itemize}
	\item $(C,\flt)$ is locally finite dimensional over $\F$ if $\dim_{\F} \flt^p C_n <\infty$ for each $p\in \Z$ and $n \in \Z$,
	\item $\flt$ is bounded below on $C_n$ if there is a $p\in\mathbb{Z}$ depending on $n$, such that $\flt^p C_n=0$.
\end{itemize}

Note that $(\flt^p C, d|_{\flt^p C})$ is a subcomplex of the chain complex $(C,d)$ for all $p \in \mathbb{Z}$. The persistent chain complex  $(PC, d_x)$ of $(C,d, \flt)$ \cite[Definition 3.1]{ComPH} is defined to be the direct sum 
\[
(PC, d_x):=\bigoplus_{p\in \mathbb{Z}} (\flt^p C, d|_{\flt^p C}).
\]
Let $x$ be a homogeneous indeterminate of degree $1$ acting on $PC$ as the natural inclusion map $i^p : \flt^p C \hookrightarrow \flt^{p+1} C$ for each $p \in\mathbb{Z}$. This makes $(PC, d_x)$ is a chain complex of graded $\Fx$-modules, where the homogeneous component of degree $p$ is $(\flt^p C, d|_{\flt^p C})$. The homology $PH:=H(PC, d_x)$ is a graded $\Fx$-module and called the persistent homology of $(C,d, \flt)$ \cite{ComPH} .
\end{defn}

Our arguments are based on Theorem \ref{Decomp}. Before stating the theorem, let us introduce the following notations.

\begin{defn} \label{def-bigrade}
Let $M$ be a graded $\Fx$ module. Then $M||n||\{s\}$ is a bigraded $\Fx$-module where
\begin{itemize}
	\item $||\cdot||$ indicates the homological grading,
	\item $\{\cdot\}$ indicates the polynomial degree shift, that is $\left( M||n||\{s\} \right)^k = M^{k-s}$.
\end{itemize}
\end{defn}
\begin{defn} \label{def-decomp-type}
Denoted by $\infZ = \mathbb{Z}_+ \cup \{+\infty\}$, we assume $n,s \in \mathbb{Z}$ and $m \in \infZ $.
Define two types of graded chain complex
\begin{align}
\U \ &= \   0 \to \fxc[n][s] \to 0   \label{eq:UF} \\
\U[m] \ &= \ 0 \to \fxc[n+1][s+m] \xrightarrow{x^m} \fxc[n][s] \to 0  \label{eq:UT}
\end{align}
where  $m = \infty$ in \eqref{eq:UF} and $0 < m < \infty$ in \eqref{eq:UT}.
\end{defn}

Our main results are the following.

\begin{thm} \label{Decomp}
Let $(PC, d_x)$ be the persistent chain complex of a filtered chain complex $(C,d, \flt)$ over $\F$ satisfying:
\begin{itemize}
    \item For each $n$, the filtration $\flt$ on $C_n$ is bounded below;
    \item $(C,d,\flt)$ is locally finite dimensional over $\F$.
    \end{itemize}
Then, up to chain homotopy and permutation of factors, $(PC,d_x)$ can be uniquely decomposed into a direct sum of graded chain complexes of types \eqref{eq:UF} and \eqref{eq:UT}.
More precisely, for each $n \in \mathbb{Z}$, there exist $K_n \in \infZ$ and a unique sequence $\{(s_n(i),m_n(i))\}_{i=1}^{K_n} \subseteq \mathbb{Z} \times \infZ $ satisfying:
\begin{itemize}
\item the set $\{i \ | \ s_n(i) = s \}$ is finite for all $s \in \Z$ and $n \in \Z$,
\item $s_n(i) \le s_n(i+1)$,
\item $m_n(i) \le m_n(i+1)$ if $s_n(i) = s_n(i+1)$,
\end{itemize} 
so that
\begin{equation}  \label{eq:DecPC}
PC \simeq \bigoplus_{n=-\infty}^\infty  \bigoplus_{i=1}^{K_n}  U_{n,s_n(i),m_n(i)} .
\end{equation}

Consequently, the persistent homology $PH=H(PC,d_x)$ can be uniquely decomposed into a direct sum of
\begin{equation} \label{eq:DecPH}
PH \cong \bigoplus_{n=-\infty}^\infty  \bigoplus_{i=1}^{K_n}  H(U_{n,s_n(i),m_n(i)}) . 
\end{equation}
Here, each factor in Decomposition \eqref{eq:DecPH} has the form $H(\U) \cong \fxc[n][s] $ or $H(\U[m]) \cong \Fx/(x^m)||n||\{s\} $ for $m < \infty$.
\end{thm}

The proof of Theorem \ref{Decomp} is included in Section \ref{sec-PH}.
Since the spectral sequences corresponding to the persistent chain complexes $\U$ and $\U[m]$ are easy to compute, Theorem \ref{Decomp} leads to a direct sum decomposition of spectral sequences over $\F$ into simple factors.  Denote by $E^{(r)}$ the $r$-th page of the spectral sequence of $(C, d,\flt)$ and by $E_{n,s}^{(r)}$ the homogeneous component of $E^{(r)}$ with homological grading $n$ and $\Fx$-module grading $s$.

\begin{thm}\label{SS-Decomp}
Let $(PC, d_x)$ be the persistent chain complex of a filtered chain complex $(C,d, \flt)$ over $\F$ satisfying:
\begin{itemize}
    \item For each $n$, the filtration $\flt$ on $C_n$ is bounded below;
    \item $(C,d,\flt)$ is locally finite dimensional over $\F$.
    \end{itemize}
Decomposition \eqref{eq:DecPH} of the persistent homology $PH=H(PC,d_x)$ in Theorem \ref{Decomp} induces a direct sum decomposition of the spectral sequence $\{E^{(r)}\}$ of $(C,d,  \flt )$. 
More precisely, for $r \in \infZ$ and the same sequence $\{(s_n(i),m_n(i))\}_{i=1}^{K_n} \subseteq \mathbb{Z} \times \infZ $ given in Theorem \ref{Decomp}, we have that,
$$E^{(r)} \cong \bigoplus_{n=-\infty}^\infty  \bigoplus_{i=1}^{K_n}  E^{(r)}(U_{n,s_n(i),m_n(i)}) ,$$
where:
\begin{itemize}
	\item $E^{(r)}(\U) \cong  \F\|n\|\{s\}$  $\forall r \ge 1$, and consequently $E^{(\infty)}(\U) \cong  \F\|n\|\{s\}$.
	\item If $m < \infty$, then 
	\[
	E^{(r)}(\U[m]) \cong  \begin{cases}
\fc[n+1][s+m] \oplus \fc[n][s] & \text{for } 1 \le r \le m, \\
0 & \text{for } r > m.
\end{cases}
\]
And, consequently $E^{(\infty)}(\U[m]) \cong  0$.
\end{itemize}
\end{thm}

Comparing the decompositions in Theorems \ref{Decomp} and \ref{SS-Decomp}, one can deduce that the spectral sequence and the persistent homology determine each other under fairly general assumptions.  
We formulate this as the following.

\begin{defn} \label{def-loc-col}
The spectral sequence $\{E^{(r)}\}$ \textit{collapses locally} to $E^{(\infty)}$ if $\{E_{n,s}^{(r)}\}$ collapses to $E_{n,s}^{(\infty)}$ for every $n,s \in \Z$.
That is, there exists $r_{n,s} \in \Z_+ $ such that $E_{n,s}^{(r)} \cong E_{n,s}^{(\infty)}$ if $r > r_{n,s}$.
\end{defn}

\begin{defn} \label{def-nu}
Let $K_n \in \infZ $ and $\{(s_n(i),m_n(i))\}_{i=1}^{K_n} \subseteq \Z \times \infZ $ be the unique sequence described in Theorem \ref{Decomp}. 
For $n,s \in \Z$ and $m \in \infZ$, define the multiplicity of $H(\U[m])$ in Decomposition \eqref{eq:DecPH} as
$$\nu_{n,s,m} := \mathlarger{|} \{ i \ | \ 1 \le i \le K_n ,\ s_n(i)=s ,\ m_n(i)=m \} \mathlarger{|}.$$
In particular, the multiplicity of $H(\U)$ is defined as
$$\nu_{n,s,\infty} := \mathlarger{|} \{ i \ | \ 1 \le i \le K_n ,\ s_n(i)=s ,\ m_n(i)=\infty \} \mathlarger{|}.$$
\end{defn}

\begin{thm} \label{SS=PH}
Let $(PC, d_x)$ and $\{E^{(r)}\}$ be the persistent chain complex and spectral sequences of a filtered chain complex $(C,d, \flt)$ over $\F$ satisfying:
\begin{itemize}
    \item For each $n$, the filtration $\flt$ on $C_n$ is bounded below;
    \item $(C,d,\flt)$ is locally finite dimensional over $\F$.
    \end{itemize}
    Then we have
\begin{enumerate}
	\item  $\dim_\F E_{n,s}^{(r)} = \nu_{n,s,\infty}+ \sum_{m \ge r} (\nu_{n,s,m} + \nu_{n-1,s-m,m})$, which shows that the persistent homology determines the spectral sequence.
	\item  The spectral sequence $\{E^{(r)}\}$ collapses locally to $E^{(\infty)}$, and
	\begin{itemize}
   \item[(a)] $\nu_{n,s,\infty}=\dim_\F E_{n,s}^{(\infty)}$,
   \item[(b)] $\nu_{n,s,m}=\dim_\F E_{n,s}^{(m)} - \dim_\F E_{n,s}^{(m+1)} - \nu_{n-1,s-m,m}$.
  \end{itemize}
	If the filtration $\flt$ is also bounded below on the entire $C$, then (a) and (b) imply that the spectral sequence determines the persistent homology, because $\nu_{n,s,m}=0$ for $s\ll 0$ since $\flt$ is bounded below uniformly.
\end{enumerate}

\end{thm}

\begin{cor}\cite[Corollary 2]{SSPH} \cite[Spectral Sequence Theorem]{ComTop}   \label{cor-SST}
Let $(PC, d_x)$ be the persistent chain complex of a filtered chain complex $(C,d, \flt)$ over $\F$ satisfying:
\begin{itemize}
    \item For each $n$, the filtration $\flt$ on $C_n$ is bounded below;
    \item $(C,d,\flt)$ is locally finite dimensional over $\F$.
    \end{itemize}
Then for any $r>0$, we have 

\begin{equation} \label{eq:dimEnr}
dim_\F E_n^{(r)}= \sum_s \left( \nu_{n,s,\infty} + \sum_{m \ge r} (\nu_{n,s,m} + \nu_{n-1,s,m}) \right).
\end{equation}

Equation \eqref{eq:dimEnr} is equivalent to the following equation from \cite[Corollary 2]{SSPH}:

\begin{equation} \label{eq:dimEnr-Basu}
dim_\F E_n^{(r)}(\flt)=  b_n(\flt) + \sum_{t-s \ge r} (\mu_n^{s,t}(\flt) + \mu_{n-1}^{s,t}(\flt))
\end{equation}
where $b_n(\flt)$ is the persistent Betti number and $\mu_n^{s,t}(\flt)$ is the persistent multiplicity that are both defined in \cite[VII.1]{ComTop}.

\end{cor}

Corollary \ref{cor-SST} follows directly from Corollary \ref{Nu=Mul} and Proposition \ref{E-ns}.

\section{Decomposition of  Persistent Homology}\label{sec-PH}

\subsection{Persistent chain complexes}

Let $(C,d,\flt)$ be a filtered chain complex over $\F$ defined in Definition \ref{def-flt-complex} and $(PC,d_x)$ be the correspond persistent chain complex.
For each $p,n \in \Z$, fix an $\F$-subspace $G^p_n$ of $\flt^p C_n$ such that $\flt^p C_n = \flt^{p-1} C_n \oplus G_n^p$. Then 
$$\flt^p C_n = \bigoplus_{q=- \infty}^p G_n^q.$$
For $q \le p$, denote by $\pi^{p,q}_n: \flt^p C_n \to G_n^q$ the natural projection.
Restricting the differential map $d_n :  C_n \to  C_{n-1}$ to $G^p_n$, we have $d_n^p : G_n^p \to \flt^p C_{n-1}= \bigoplus_{q=- \infty}^p G_{n-1}^q$. 
Define the map $d_n^{p,q} : G_n^p \to G_{n-1}^q$ by $d_n^{p,q} = \pi_{n-1}^{p,q} \circ d_n^p$ for any $p \ge q$.
Thus $d_n^p = \sum_{q \le p} d_n^{p,q}$.

\begin{defn} \label{def-pc}
For each $n \in \Z$, we define a graded $\Fx$-module as ${\PC_n = \Fx \tsr_\F C_n }$ where $\PC^p_n := \bigoplus_{k \ge 0} x^k G^{p-k}_n$ is the homogeneous component of $\PC_n$ of degree $p$.
$\PC_n$ is a graded $\Fx$-module for each $n$. Define an $\Fx$-module map $(\d_x)_n : \PC_n \to \PC_{n-1}$ such that $(\d_x)_n^p := (\d_x)|_{G_n^p} : G^p_n \to \PC^p_{n-1}$ is given by $(\d_x)_n^p=\sum_{q \le p} x^{p-q}d_n^{p,q}$.
\end{defn}

Note that $(\d_x)_n$ is a homogeneous map preserving the module grading since the degree of $d_n^{p,q}$ is $q-p$.

\begin{prop} \label{PC-iso}
$(\PC,\d_x)$ is isomorphic to $(PC, d_x)$. In particular, $(PC, d_x)$ is a free $\Fx $-module.
\end{prop}

\begin{proof}

Recall that $\flt^p C_n = \bigoplus_{q=- \infty}^p G^p_n$ and $(PC, d_x)=\bigoplus_{p\in \mathbb{Z}} (\flt^p C, d|_{\flt^p C})$.
We have $PC_n = \bigoplus_{p \in \Z} \flt^p C_n =  \bigoplus_{p \in \Z} \left( \bigoplus_{q=- \infty}^p G^q_n \right)$. 
For any $n \in \Z$ and any $u \in \flt^p C_n \subset PC_n $, there is a decomposition of $u$ given by $u = \sum_{q \le p} g_n^{q,p}$ where each $g_n^{q,p} \in G_n^q \subset \flt^p C_n$. 
Define $\varphi_n: PC_n \to \PC_n$ by 
$$\varphi_n(u) = \varphi_n \left( \sum_{q \le p} g_n^{q,p} \right) = \sum_{q \le p} x^{p-q} g_n^q$$ 
where $g_n^q = g_n^{q,p}$ when viewed as an element of $ G_n^q \subset \PC_n$.
This is an $\F$-linear isomorphism with $\varphi_n^{-1} \left(  \sum_{q \le p} x^{p-q} g_n^q \right) =  \sum_{q \le p} g_n^{q,p}$.
Consider $i^p: \flt^p C \to \flt^{p+1} C$.
$g_n^{q,p} \in G_n^q \subset \flt^p C_n$ and $g_n^{q,p+1} \in G_n^q \subset \flt^{p+1} C_n$ are the same element as  $ g_n^p \in G^p_n \subset \PC_n$. 
Then $i^p ( g_n^{q,p}) =  g_n^{q,p+1}$ and, for $u \in \flt^pC_n$,
$$\begin{aligned}
\varphi_n (xu) & = \varphi(i^p(u)) =   \varphi_n \left(   i^p \left(\sum_{q \le p} g_n^{q,p} \right) \right) = \varphi_n \left(   \sum_{q \le p} g_n^{q,p+1} \right) \\
& =  \sum_{q \le p} x^{(p+1)-q} g_n^q =  \sum_{q \le p} x^{p-q+1} g_n^q \\
& = x \cdot  \sum_{q \le p} x^{p-q} g_n^q = x \cdot \varphi_n(u).
\end{aligned}$$
To verify $\varphi$ is a chain map, we have
$$\begin{aligned}
\d_x \circ \varphi_n (u) &= \d_x \left(  \sum_{q \le p} x^{p-q} g_n^q \right) =     \sum_{q \le p} x^{p-q} \cdot (\d_x)_n^q ( g_n^q ) \\ 
& =  \sum_{q \le p} x^{p-q} \sum_{k \le q} x^{q-k} d_n^{q,k}(g^p_n) =  \sum_{q \le p}  \sum_{k \le q} x^{(p-q)+(q-k)} d_n^{q,k}(g^p_n) \\
& =  \sum_{q \le p}  \sum_{k \le q} x^{p-k} d_n^{q,k}(g^p_n).
\end{aligned}$$
Note that $d^{q,k}_n(g_n^{q,p}) \in G_{n-1}^k \subset \flt^p C_{n-1}$. 
We also have
$$\begin{aligned}
\varphi_{n-1} \circ d_x (u) & = \varphi_{n-1} \left(  d|_{\flt^p C_n} \left( \sum_{q \le p} g_n^{q,p}  \right) \right)  = \varphi_{n-1} \left(  \sum_{q \le p} d^q_n(g_n^{q,p})  \right) \\
& =\varphi_{n-1} \left(  \sum_{q \le p} \sum_{k \le q}  d^{q,k}_n(g_n^{q,p})  \right) =  \sum_{q \le p} \sum_{k \le q} \varphi_{n-1} \left( d^{q,k}_n(g_n^{q,p})  \right) \\
& =  \sum_{q \le p}  \sum_{k \le q} x^{p-k} d_n^{q,k}(g^p_n).
\end{aligned}$$
Thus $\d_x \circ \varphi = \varphi \circ d_x $ and $\varphi$ is a chain map. 

Since $\d_x = \varphi \circ d_x \circ \varphi^{-1}$ and $d_x \circ d_x =0$, we have $\d_x \circ \d_x = \varphi \circ d_x \circ d_x \circ \varphi^{-1} = 0$.
Note that $\d_x$ preserves the module grading.
Therefore $(\PC,\d_x)$ is a chain complex of graded $\Fx$-modules and isomorphic to $(PC,d_x)$ induced by $\varphi$.
\end{proof}

Proposition \ref{PC-iso} shows that $(\PC, \d_x)$ is isomorphic to persistent chain complex $(PC,d_x)$. We will no longer distinguish between these two in the rest of this manuscript.

\subsection{Existence of the decomposition}

\begin{defn} \label{def-loc-finite}
A graded $\F $-space $M$ is locally finite dimensional if $M^k$ is finite dimensional over $\F$ for every $k \in \Z$, where $M^k$ is the $\F $-subspace of $M$ consisting of all homogeneous elements of degree $k$.
\end{defn}

\begin{lem} \label{LemGrading}
Let $\Fx$ be the graded polynomial ring with $deg(x)=1$. Assume that $M$ is a free graded $\mathbb{F}[x]$-module satisfying:
\begin{itemize}
\item   $\Fx$-grading of $M$ is bounded below,
   
    \item $M$ is locally finite dimensional over $\F$.
   \end{itemize}
Then $M$ admits a countable homogeneous basis.

Denote by $\lambda_M^k$ the $\mathbb{F}$-dimension of $(M/xM)^k$, where $(M/xM)^k$ is $\mathbb{F}$-subspace of $M/xM$ consisting of all elements of grading $k$. Assume that $g$ is the lowest $\Fx $-grading of $M$. Define the sequence $L_M$
as\[L_M=(\underbrace{g,\cdots, g}_{\lambda_M^g},\underbrace{g+1,\cdots, g+1}_{\lambda_M^{g+1}}, \cdots, \underbrace{g+k,\cdots, g+k}_{\lambda_M^{g+k}} , \cdots).\]
Let $E$ be a countable collection of homogeneous elements of M such that the cardinality of $E\cap M^k$ is finite for every $k$. After arranging the elements of $E$ in non-decreasing order by their degrees, the degrees of elements in $E$ form a non-decreasing sequence $L_E$. Then $L_B=L_M$ for any homogeneous basis $B$ of $M$.
   
\end{lem}
\begin{proof}
 $M$ has a homogeneous basis since free graded modules over polynomial rings with grading bounded below admit homogeneous bases \cite[Lemma 4.4]{MR3234803} \cite[Page 130, Exercise 3]{MR1096302}.
 \par
 Let $B$ be any homogeneous basis for $M$. To prove that $L_B=L_M$, it suffices to show that $\pi(B)$ is a homogeneous $\mathbb{F}$-basis for $(M/xM)$, where $\pi:M\xrightarrow{}(M/xM)$ is the canonical projection. Since $\pi$ is surjective, $\pi(B)$ spans $(M/xM)$. Assume that $\pi(B)$ is not $\mathbb{F}$-linearly independent, then there exist $P>0$ with $ b_p \in B$, $f_p\in \mathbb{F}$ and $g_p\in \mathbb{F}[x]$ for all $1 \le p \le P$, such that
 
\begin{equation} \label{eq:LG}
 \sum_{p=1}^P(f_p-g_px)b_p=0
\end{equation} 
and $f_p-g_px \neq 0$ for each $p$. By the $\mathbb{F}[x]$-linearly independence of $B$, all coefficients in equation \eqref{eq:LG} are zero. We have that $f_p-xg_q=0$, which implies that $f_p$ and $g_p$ are both equal to zero for all $p$. This is a contradiction.    

\end{proof}

\begin{lem} \label{LemDec}
Assume that\ $$\begin{tikzcd}
\cdots \arrow[r] & M \arrow[r, "d"] & N \arrow[r]  & \cdots
\end{tikzcd}$$
is a chain complex of free graded $\mathbb{F}[x]$-module satisfying:
\begin{itemize}
    \item Its differential $d$ is homogeneous and preserves the $\Fx $-grading;
    \item The $\mathbb{F}[x]$-grading of $M$ and $N$ are bounded below;
    \item $M$ and $N$ are locally finite dimensional over $\F$.
    \end{itemize}
Then $M$ has a homogeneous $\Fx$-basis $V \sqcup W$ and $N$ has a homogeneous $\Fx$-basis $\overline{W} \sqcup R$ satisfying 
 \begin{itemize}
    \item[(I)~] $d(v)=0$ for every $v\in V$;
    \item[(II)] There exists a bijection  $\bar{~}:~W \to \overline{W}$ mapping each $w \in W$ to a $\bar{w} \in \overline{W}$ such that $d(w) = x^m \bar{w}$ for some non-negative integer $m$ depending on $w$;
    \end{itemize}
Considering the submodules $M^{V}$, $M^{W}$, $N^{\overline{W}}$ and $N^R$ generated by $V$, $W$, $\overline{W}$, and $R$, the chain complex decomposes into a direct sum of three subchain complexes as
$$\begin{tikzcd}
\cdots \arrow[r] & M^{V} \arrow[r, "d" {name=L1}] & 0   &  \\
0 \arrow[r] & M^{W} \arrow[r, "d" {name=L2}] \arrow[u, from=L2, to=L1, phantom, "\bigoplus" description] & N^{\overline{W}} \arrow[r]  & 0 \\
 & 0 \arrow[r, "d" {name=L3}]  & N^R \arrow[r] \arrow[d, from=L2, to=L3, phantom, "\bigoplus" description] & \cdots
\end{tikzcd}$$

\end{lem}
    
\begin{proof}
We prove this by induction. 
Let $B^{(0)}$ be a homogeneous basis for $M$ and $\Tilde{B}^{(0)}$ be a homogeneous basis for $N$, respectively. 
By Lemma \ref{LemGrading}, $L_M$ is the non-decreasing sequence of degrees in any homogeneous basis of $M$.
Let sets $W^{(0)}$, $V^{(0)}$ and $\overline{W}^{(0)}$ be empty sets. Let $Z^{(0)} = B^{(0)}$ and $R^{(0)} = \Tilde{B}^{(0)}$. 
We will construct by induction the sets $B^{(l)}$, $\Tilde{B}^{(l)}$, $V^{(l)}$, $W^{(l)}$, $O^{(l)} $, $Z^{(l)}$, $\overline{W}^{(l)}$ and $R^{(l)}$, satisfying the following conditions:
 \begin{itemize}
 	\item[(i)] $B^{(l)}$ is a homogeneous basis for $M$ and $\Tilde{B}^{(l)}$ is a homogeneous basis for $N$;
 	\item[(ii)] $O^{(l)} = V^{(l)} \sqcup W^{(l)}$,  $ B^{(l)} = Z^{(l)} \sqcup O^{(l)}$ and $\Tilde{B}^{(l)} = \overline{W}^{(l)} \sqcup R^{(l)}$;
 	\item[(iii)] $V^{(l-1)} \subset V^{(l)}$, $W^{(l-1)} \subset W^{(l)}$, $\overline{W}^{(l-1)} \subset \overline{W}^{(l)}$ and $R^{(l-1)} \supset R^{(l)}$;
    \item[(iv)] $d(Z^{(l)}) \subset \mathrm{Span}_{\Fx } R^{(l)}$;
    \item[(v)] $d(v)=0$ for every $v\in V^{(l)}$;
    \item[(vi)] There exists a bijection  $\bar{~}:~W^{(l)} \to \overline{W}^{(l)}$ mapping each $w \in W^{(l)}$ to a $\bar{w} \in \overline{W}^{(l)}$ such that $d(w) = x^m \bar{w}$ for some positive integer $m$ depending on $w$;
    \item[(vii)] The sequence $L_{O^{(l)}}$ for set of homogeneous elements\footnote{Notation follows Lemma \ref{LemGrading}.} is consist of the first $l$ terms of $L_M$.
    \end{itemize}
In each iteration, we remove an element from $Z^{(l)}$ with the lowest grading and add it to $O^{(l)}$.
We will induct on the size of organized base set $O^{(l)}$. 
Conditions (i)-(vii) are trivially true when $l=0$.
Then assume that we have updated all these sets satisfying conditions (i)-(vii) after $l$ steps. 
By (vii), $|O^{(l)}| = |V^{(l)}|  + |W^{(l)}| = l$. We will then construct $B^{(l+1)}$, $\Tilde{B}^{(l+1)}$, $V^{(l+1)}$, $W^{(l+1)}$, $\overline{W}^{(l+1)}$, $O^{(l+1)} $, $Z^{(l+1)}$ and $R^{(l+1)}$.  

 \par
Let $w$ be an element of $Z^{(l)}$ with the lowest module grading, whose degree by Lemma \ref{LemGrading} is the $(l+1)$-th term of $L_M$. 

If $d(w)=0$, we will do the following updates: 
\[ V^{(l+1)}=V^{(l)} \cup \{w\},\  W^{(l+1)} = W^{(l)},\ \overline{W}^{(l+1)} = \overline{W}^{(l)}, \]
\[ Z^{(l+1)}=Z^{(l)} \setminus \{w\} ,\ O^{(l+1)}=V^{(l+1)} \cup W^{(l+1)} ,\  R^{(l+1)} = R^{(l)}, \]
\[ B^{(l+1)} = Z^{(l+1)} \cup O^{(l+1)} ,\ \Tilde{B}^{(l+1)} = \overline{W}^{(l+1)} \cup R^{(l+1)}. \]
This update preserves condition (v) and automatically satisfies all other conditions.

 If $d(w)\neq0$, assume that $d(w)=\sum_{e \in E} \alpha_e e$ where $E$ is a finite subset of $R^{(l)}$ by (iv), and each $\alpha_e=c_e x^{m_e}$ is a non-zero monomial in $\Fx$.  
Let $\hat{e}$ be an element with the highest module grading in $E$. Note that $m_e=\deg(w)-\deg(e)$ and therefore $m_{\hat{e}} = \min_{e \in E} m_e$. 
Define
  \[\bar{w}=\frac{d(w)}{ x^{m_{\hat{e}}}}=c_{\hat{e}} \hat{e}+\sum_{e \in E, e \neq \hat{e}}c_e x^{m_e-m_{\hat{e}}}e.\] 
  and do the following updates:
\[V^{(l+1)}=V^{(l)},\  W^{(l+1)} = W^{(l)} \cup \{w\},\ \overline{W}^{(l+1)} = \overline{W}^{(l)} \cup \{\bar{w}\},  \]
\[R^{(l+1)} = R^{(l)}\setminus \{\hat{e}\},\ O^{(l+1)}=V^{(l+1)} \cup W^{(l+1)} ,\ \Tilde{B}^{(l+1)}=R^{(l+1)} \cup \overline{W}^{(l+1)}. \]
Then $\Tilde{B}^{(l+1)}$ is still a homogeneous basis for $N$.
Thus the bijection $\ \bar{•}\ $ in (vi) is extend by $w \mapsto \bar{w}$. 
Note that $d(w) = x^m \bar{w}$.
It remains to construct $Z^{(l+1)}$ and show it satisfies condition (iv).

By conditions (iv) in the previous step,  any element in $\overline{W}^{(l)}$ will not appear in the image of $z \in Z^{(l)}$ since $z$ is not in $W^{(l)}$ according to condition (ii).
Now consider for each $z$ in $Z^{(l)}\setminus \{w\}$, then $d(z)$ is a linear combination of base elements in $R^{(l)}$.
We have $d(z)=\beta_{z}\bar{w}+\sum_{e \in E_z} \beta_e e$, where $\beta$ are the coefficient monomials and $E_z$ is a finite subset of $R^{(l+1)}$. 
Notice that $\deg(\beta_z)=\deg(z)-\deg(\bar{w})\geq \deg(w)- \deg(\bar{w})=m_{\hat{e}}$. 
Define
\[z'=z-\frac{\beta_z}{x^{m_{\hat{e}}}}w.\]
Then $d(z')=d(z)-\frac{\beta_z}{x^{m_{\hat{e}}}}d(w)=d(z)-\beta_z \bar{w}=\sum_{e \in E_z} \beta_e e$ and therefore $d(z') \in \mathrm{Span}_{\Fx } R^{(l+1)}$. 
Now we update $Z^{(l+1)}$ and $B^{(l+1)}$ by
\[ Z^{(l+1)}=\{z'|z \in Z^{(l)} \setminus \{w\} \} \]
and
\[ B^{(l+1)} = Z^{(l+1)} \cup O^{(l+1)} .\]
Hence  $d(Z^{(l+1)}) \subset \mathrm{Span}_{\Fx } R^{(l+1)}$. $B^{(l+1)}$ is obtained from $B^{(l)}$ by an invertible $\Fx $-linear transformation, so $B^{(l+1)}$ is also a basis for $M$. This proves condition (iv). 
Note that $|O^{(l+1)}| = |V^{(l)}|  + |\{w\}| + |W^{(l)}|= l+1$. Therefore conditions (i)-(vi) work for all $B^{(l)}$, $\Tilde{B}^{(l)}$, $V^{(l)}$, $W^{(l)}$, $\overline{W}^{(l)}$, $O^{(l)} $, $Z^{(l)}$ and $R^{(l)}$ for each natural number $l$.

Recall that $M^k$ is finite dimensional for every $k$. Let $g$ be the lowest grading of $M$, which is the first term in $L_M$. Thus $\bigoplus_{k=g}^K M^k$ is finite dimensional for any $K \ge g$. 
Set $l_K= \dim_\F \bigoplus_{k=g}^K M^k = \sum_{k=g}^K \lambda_M^k$ where $\lambda_M^k$ is given in Lemma \ref{LemGrading}. 
Observe that $|O^{(l_K)}| = l_K$ and 
\[ \bigoplus_{k=g}^K M^k \subseteq \mathrm{Span}_{\Fx } O^{(l_K)} \]
since $Z^{(l_K)}$ contains only elements that have degree larger than $K$. 
Thus
\[ M = \bigoplus_{k=g}^\infty M^k \subseteq \mathrm{Span}_{\Fx } \bigcup_{l=0}^\infty O^{(l)}. \]
Then any element in $M$ can be written as a linear combination of elements in some $O^{(l)}$.
Let \[V=\bigcup_{l=0}^\infty V^{(l)} \ \text{and} \ W=\bigcup_{l=0}^\infty W^{(l)},\] then this is our desired $V$ and $W$. 
Note that $V \sqcup W=\bigcup_{l=0}^\infty O^{(l)}$.
Each $O^{(l)}$ is a subset of a homogeneous basis $B^{(l)}$ and therefore linearly independent. 
Hence $V \sqcup W$ is indeed a homogeneous basis for $M$.

In each step, we have constructed a bijection between $W^{(l)}$ and $\overline{W}^{(l)}$. Define 
\[ \overline{W}=\bigcup_{l=0}^\infty \overline{W}^{(l)} \ \text{and} \ R = \bigcap_{l=0}^\infty R^{(l)} .  \]
Let $g'$ be the lowest module grading of $N$.
Fix a $K \ge g'$, $\tilde{B}^{(l)} \cap \left( \bigoplus_{k=g'}^K N^k \right)$ is finite for every $l$ since $N$ is locally finite dimensional.
Note that $R^{(l-1)} \supset R^{(l)}$ for each $l$ from condition (iii). 
Then for there exists a positive integer $l_0$ depending on $K$, such that
$$R^{(l_0)} \cap \left( \bigoplus_{k=g'}^K N^k \right) = \left( \bigcap_{l=0}^\infty R^{(l)} \right) \cap \left( \bigoplus_{k=g'}^K N^k \right) = R \cap \left( \bigoplus_{k=g'}^K N^k \right).$$
Therefore $\left( \bigoplus_{k=g'}^K N^k \right)$ has a basis $\tilde{B}^{(l_0)} \cap \left( \bigoplus_{k=g'}^K N^k \right) = ( \overline{W} \sqcup R) \cap \left( \bigoplus_{k=g'}^K N^k \right)$.
$\overline{W} \sqcup R$ spans $N$ as $K$ growing. Hence $\overline{W} \sqcup R$ is indeed a homogeneous basis for $N$.
This concludes the constructions of $V$, $W$, $\overline{W}$ and $R$.

Notice that after the construction of $\overline{W}$ and $R$, we obtained a subchain complex
$$\begin{tikzcd} 
0 \arrow[r, "d" ]  & N^R \arrow[r]  & \cdots
\end{tikzcd}$$
where $N^R$ is the submodule generated by $R$ in $N$.
Denote by $N'$ be the term before $M$ in the chain complex and consider the other direct summand
$$\begin{tikzcd} 
\cdots \ar[r] & N' \ar[r,"d"] & M \arrow[r, "d" ]  & N^{\overline{W}} \arrow[r]  & 0.
\end{tikzcd}$$
If $d(u) \neq 0$ for any $u \in N'$, then $d(u)=\sum_{v \in V} \gamma_v v + \sum_{w \in W} \gamma_w w$ where all but finitely many $\gamma_v$ and $\gamma_w$ are $0$. 
We claim that $\gamma_w = 0$ for each $w$. Otherwise, since $d(v)=0$ for each $v \in V$ and $d(w) = x^{m_w} \bar{w}$ for each $w \in W$, we have
$$0=(d \circ d)(u)=\sum_{w \in W} \gamma_w d(w) = \sum_{w \in W} \gamma_w x^{m_w} \bar{w}.$$ 
Since each $\bar{w} \in \overline{W}$ is an $\Fx $-base element, $\gamma_w$ must be $0$ if $w \in W$. Therefore we will further decompose the subchain complex into
$$\begin{tikzcd}
\cdots \arrow[r] & M^{V} \arrow[r, "d" ] & 0 
\end{tikzcd}$$
where $M^V$ is the submodule in $M$ generated be $V$, and 
$$\begin{tikzcd}
0 \arrow[r] & M^{W} \arrow[r, "d"]  & N^{\overline{W}} \arrow[r]  & 0
\end{tikzcd}$$
where $M^{W}$ is the submodule in $M$ generated by $W$ as well as  $N^{\overline{W}}$ is the submodule in $N$ generated by ${\overline{W}}$.

\end{proof}

\begin{prop} \label{ThmDec}
Assume that 
$$(A,d)=\begin{tikzcd}
\cdots \arrow[r, "d_{n+1}"] & A_n \arrow[r, "d_n"] & A_{n-1} \arrow[r, "d_{n-1}"] & \cdots
\end{tikzcd}$$ 
is a chain complex of free graded $\mathbb{F}[x]$-module satisfying:
\begin{itemize}
    \item Its differential $d$ is homogeneous and preserves the $\Fx $-grading of $A$;
    \item For each $n$, the $\mathbb{F}[x]$-grading of $A_n$ is bounded below;
    \item For each $n$, $A_n$ is locally finite dimensional over $\F$.
    \end{itemize}
Then every $A_n$ has a homogeneous $\Fx$-basis that is the disjoint union of three subsets $V_n$, $W_n$ and $\overline{W}_n$, satisfying the following conditions:
 \begin{itemize}
    \item[(I)~] $d(v)=0$ for every $v\in V_n$;
    \item[(II)] There exists a bijection  $\bar{~}:~W_n \to \overline{W}_{n-1}$ mapping each $w \in W_n$ to a $\bar{w} \in \overline{W}_{n-1}$ such that $d(w) = x^m \bar{w}$ for some non-negative integer $m$ depending on $w$.

 \end{itemize}
    
\end{prop}

\begin{proof}

By Lemma \ref{LemGrading}, each $A_n$ has a homogeneous basis.
Since for each $n$, $A_n$ satisfies all the conditions in Lemma \ref{LemDec},  we may apply this lemma to
\[\begin{tikzcd}
\cdots \ar[r, "d_2"]  & A_1 \arrow[r, "d_1"] & A_0 \arrow[r, "d_0"] & \cdots
\end{tikzcd}\]
We will obtain $\Fx $-basis sets $\Tilde{V}_1 \sqcup W_1$ for $A_1$ and $\overline{W}_0 \sqcup R_0$ for $A_0$ with a direct sum decomposition
$$\begin{tikzcd}
\cdots \arrow[r, "d_2"] & \Tilde{A}_1 \arrow[r, "" {name=L1}] & 0   &  \\
0 \arrow[r] & \cU_1 \arrow[r, "d_1" {name=L2}] \arrow[u, from=L2, to=L1, phantom, "\bigoplus" description] & \overline{\cU}_0 \arrow[r]  & 0 \\
 & 0 \arrow[r, "" {name=L3}]  & \Tilde{A}_0 \arrow[r, "d_0"] \arrow[d, from=L2, to=L3, phantom, "\bigoplus" description] & \cdots
\end{tikzcd}$$
Here $\Tilde{A}_1$,  $\cU_1$,  $\overline{\cU}_0$ and  $\Tilde{A}_0$ are submodules generated by sets $\Tilde{V}_1$, $W_1$,  $\overline{W}_0$ and $R_0$ respectively.

We will prove this proposition by two inductions separately on the upward branch
\[\begin{tikzcd}
\cdots \ar[r, "d_3"]  & A_2 \arrow[r, "d_2"] & \Tilde{A}_1 \arrow[r] & 0
\end{tikzcd}\]
and the downward branch
\[\begin{tikzcd}
0 \ar[r]  & \Tilde{A}_0 \arrow[r, "d_0"] & A_{-1} \arrow[r, "d_{-1}"] & \cdots
\end{tikzcd}\]

For the upward branch, we first apply Lemma \ref{LemDec} to 
\[\begin{tikzcd}
\cdots \ar[r, "d_3"]  & A_2 \arrow[r, "d_2"] & \Tilde{A}_1 \arrow[r] & 0.
\end{tikzcd}\]
It decomposes into
$$\begin{tikzcd}
\cdots \arrow[r, "d_3"] & \Tilde{A}_2 \arrow[r, "" {name=L1}] & 0   &  \\
0 \arrow[r] & \cU_2 \arrow[r, "d_2" {name=L2}] \arrow[u, from=L2, to=L1, phantom, "\bigoplus" description] & \overline{\cU}_1 \arrow[r]  & 0 \\
 & 0 \arrow[r, "" {name=L3}]  & F_1 \arrow[r] \arrow[d, from=L2, to=L3, phantom, "\bigoplus" description] & 0
\end{tikzcd}$$
where $\Tilde{A}_2$,  $\cU_2$,  $\overline{\cU}_1$ and  $F_1$ are the submodules generated by sets $\Tilde{V}_2$, $W_2$,  $\overline{W}_1$ and $R_1$ respectively. 
Note that $\Tilde{V}_2 \sqcup W_2$ is a basis for $A_2 = \Tilde{A}_2 \oplus \cU_2$ and $\overline{W}_1 \sqcup R_1$ is a basis for $\Tilde{A}_1 = \overline{\cU}_1 \oplus F_1$. $R_1$ satisfies condition (I) in $A_1$ thus let $V_1=R_1$. 
Now we have obtained $V_1$, $W_1$ and $\overline{W}_1$. 
Since $W_1$ generates $\cU_1$ and $\overline{W}_1 \sqcup V_1$ is a basis for $\Tilde{A}_1$, $V_1 \sqcup W_1 \sqcup \overline{W}_1$ is a basis for $A_1 = \cU_1 \oplus \Tilde{A}_1$, satisfying conditions (I) and (II).

Then assume that $p > 1$ and we have obtained all $V_k$, $W_k$ and $\overline{W}_k$ for $p > k \ge 1$. 
Also, $W_p$ generates $\cU_p$ and $\Tilde{V}_p$ generates $\Tilde{A}_p$ while $W_p \sqcup \Tilde{V}_p$ being a basis for $A_p = \cU_p \oplus \Tilde{A}_p$.
Then it suffices to obtain $V_p$ and $\overline{W}_p$ from $\Tilde{V}_p$. 
By applying Lemma \ref{LemDec}, the chain complex
$$\begin{tikzcd}
 \cdots \ar[r] & A_{p+1} \arrow[r, "d_{p+1}"] & \Tilde{A}_p \arrow[r, "d_p"] & 0  
 \end{tikzcd}$$
decomposes into
$$\begin{tikzcd}
\cdots \arrow[r, "d_{p+2}"] & \Tilde{A}_{p+1} \arrow[r, "" {name=L1}] & 0   &  \\
0 \arrow[r] & \cU_{p+1} \arrow[r, "d_{p+1}" {name=L2}] \arrow[u, from=L2, to=L1, phantom, "\bigoplus" description] & \overline{\cU}_p \arrow[r]  & 0 \\
 & 0 \arrow[r, "" {name=L3}]  & F_p \arrow[r] \arrow[d, from=L2, to=L3, phantom, "\bigoplus" description] & 0
\end{tikzcd}$$
where $\Tilde{A}_{p+1}$,  $\cU_{p+1}$,  $\overline{\cU}_p$ and  $F_p$ are the submodules generated by sets $\Tilde{V}_{p+1}$, $W_{p+1}$,  $\overline{W}_p$ and $R_p$ respectively.
Note that $\Tilde{V}_{p+1} \sqcup W_{p+1}$ is a basis for $A_{p+1} = \Tilde{A}_{p+1} \oplus \cU_{p+1}$ and $\overline{W}_{p} \sqcup R_{p}$ is a basis for $\Tilde{A}_{q} = \overline{\cU}_p \oplus F_p$. Since $d_p$ maps $\Tilde{A}_{p}$ to $0$, we assign $V_p = R_p$ satisfying condition (I). Hence we have obtained $V_p \sqcup W_p \sqcup \overline{W}_p$ as a basis for $A_p= \cU_p \oplus \Tilde{A}_p$, satisfying conditions (I) and (II).

For the downward branch, we first apply Lemma \ref{LemDec} to
\[\begin{tikzcd}
0 \ar[r]  & \Tilde{A}_0 \arrow[r, "d_0"] & A_{-1} \arrow[r, "d_{-1}"] & \cdots
\end{tikzcd}\]
It decomposes into
$$\begin{tikzcd}
0 \arrow[r] & F_0 \arrow[r, "" {name=L1}] & 0   &  \\
0 \arrow[r] & \cU_0 \arrow[r, "d_0" {name=L2}] \arrow[u, from=L2, to=L1, phantom, "\bigoplus" description] & \overline{\cU}_{-1} \arrow[r]  & 0 \\
 & 0 \arrow[r, "" {name=L3}]  & \Tilde{A}_{-1} \arrow[r, "d_{-1}"] \arrow[d, from=L2, to=L3, phantom, "\bigoplus" description] & \cdots
\end{tikzcd}$$
where $F_0$,  $\cU_0$,  $\overline{\cU}_{-1}$ and  $\Tilde{A}_{-1}$ are the submodules generated by sets $V_0$, $W_0$,  $\overline{W}_{-1}$ and $R_{-1}$ respectively. 
Note that $V_0 \sqcup W_0$ is a basis for $\Tilde{A}_0 = F_0 \oplus \cU_0$ and $\overline{W}_{-1} \sqcup R_{-1}$ is a basis for $A_{-1} = \overline{\cU}_{-1} \oplus \Tilde{A}_{-1}$. 
Recall that $\overline{W}_0$ is obtained at the beginning of the proof, thus now we have obtained all $V_0$, $W_0$ and $\overline{W}_0$. 
Since $\overline{W}_0$ generates $\overline{\cU}_0$ and $V_0 \sqcup W_0$ is a basis for $\Tilde{A}_0$, $V_0 \sqcup W_0 \sqcup \overline{W}_0$ is a basis for $A_0= \overline{\cU}_0 \oplus \Tilde{A}_0$, satisfying conditions (I) and (II).

Then assume that $q < 0$ and we have obtained all $V_k$, $W_k$ and $\overline{W}_k$ for $0 \le k < q$. 
Also, $\overline{W}_q$ generates $\overline{\cU}_q$ and $R_q$ generates $\Tilde{A}_q$ while $\overline{W}_q \sqcup R_q$ being a basis for $A_q = \overline{\cU}_q \oplus \Tilde{A}_q$.
Then it suffices to obtain $V_q$ and $W_q$ from $R_q$. 
By applying Lemma \ref{LemDec}, the chain complex
$$\begin{tikzcd}
0 \arrow[r]  & \Tilde{A}_{q} \arrow[r, "d_{q}"] & A_{q-1} \ar[r, "d_{q-1}"]  & \cdots
\end{tikzcd}$$
decomposes into
$$\begin{tikzcd}
0 \arrow[r] & F_q \arrow[r, "" {name=L1}] & 0   &  \\
0 \arrow[r] & \cU_{q} \arrow[r, "d_{q}" {name=L2}] \arrow[u, from=L2, to=L1, phantom, "\bigoplus" description] & \overline{\cU}_{q-1} \arrow[r]  & 0 \\
 & 0 \arrow[r, "" {name=L3}]  & \Tilde{A}_{q-1} \arrow[r, "d_{q-1}"] \arrow[d, from=L2, to=L3, phantom, "\bigoplus" description] & \cdots
\end{tikzcd}$$
where $F_q$,  $\cU_q$,  $\overline{\cU}_{q-1}$ and  $\Tilde{A}_{q-1}$ are the submodules generated by sets $V_q$, $W_q$,  $\overline{W}_{q-1}$ and $R_{q-1}$ respectively. 
Note that $V_q \sqcup W_q$ is a basis for $\Tilde{A}_q = F_q \oplus \cU_q$ and $\overline{W}_{q-1} \sqcup R_{q-1}$ is a basis for $A_{q-1} = \overline{\cU}_{q-1} \oplus \Tilde{A}_{q-1}$. 
Hence we have obtained $V_q \sqcup W_q \sqcup \overline{W}_q$ as a basis for $A_q = \overline{U}_q \oplus \Tilde{A}_q$, satisfying conditions (I) and (II).

In conclusion, we have obtained $V_n \sqcup W_n \sqcup \overline{W}_n$ be a homogeneous basis of $A_n$ satisfying condition (I) and (II) for each $n$.
\end{proof}

\begin{cor} \label{UDec}
The chain complex $(A,d)$ in Proposition \ref{ThmDec} is homotopic to the direct sum of chain complexes of the Types \eqref{eq:UF} and \eqref{eq:UT} given in Definition \ref{def-decomp-type}.
Each $v \in V_n$ generates a component $U_{n, \deg v, \infty}$. 
Each pair $w \in W_{n+1}$ and $\bar{w} \in \overline{W}_n$ generate a component $U_{n, \deg \bar{w}, m_w}$ where $d(w) = x^{m_w} \bar{w}$ for some positive integer $m_w$ depending on $w$. 

Recall that $\infZ = \mathbb{Z}_+ \cup \{+\infty\}$. 
Then for each $n$, there exist $K_n \in \infZ$ and the sequence $\{(s_n(i),m_n(i))\}_{i=1}^{K_n} \subseteq \mathbb{Z} \times \infZ $ satisfies $s_n(i) \le s_n(i+1)$ and $m_n(i) \le m_n(i+1)$ if $s_n(i) = s_n(i+1)$, such that
$$A \simeq \bigoplus_{n=-\infty}^\infty  \bigoplus_{i=1}^{K_n}  U_{n,s_n(i),m_n(i)} .$$
Consequently,
$$H(A) \cong \bigoplus_{n=-\infty}^\infty  \bigoplus_{i=1}^{K_n}  H(U_{n,s_n(i),m_n(i)}) $$
where $H(\U) \cong \fxc[n][s] $ and $H(\U[m]) \cong \Fx/(x^m)||n||\{s\}$.
 
\end{cor} 
 
\begin{proof}
Let $V_n \sqcup W_n \sqcup \overline{W}_n$ be the basis of $A_n$ from Proposition \ref{ThmDec}.
Note that each $v \in V_n$ generates a subcomplex $\U $ and each pair of $(w,\bar{w}) \in W_{n+1} \times W_n$ generate a subcomplex $\U[m] $. 
So it suffices to show that 
\[U_{n,s,0} \ = \  0 \to \fxc[n+1][s] \xrightarrow{1} \fxc[n][s] \to 0 \]
is homotopic to zero. Consider the identity map $id$ and the chain homotopy given by $h=1 \cdot$ in the diagram
\[ \begin{tikzcd}[column sep = large]
0 \ar[r] & \Fx \ar[r,"1"] & \Fx \ar[r] & 0 \\
0 \ar[r] \ar[ur, leftarrow,"0"]  & \Fx \ar[r,"1"] \ar[u,leftarrow,"id"] \ar[ur, leftarrow,"1"]  & \Fx \ar[r] \ar[u,leftarrow,"id"] \ar[ur, leftarrow,"0"] & 0 
\end{tikzcd}\]
Here $h \circ d + d \circ h = id$. Hence $id$ is null-homotopic and 
\[U_{n,s,0} \ = \  0 \to \fxc[n+1][s] \xrightarrow{1} \fxc[n][s] \to 0 \]
is homotopic to zero when.
\end{proof}

\subsection{Uniqueness of the decomposition}
The following lemma is a generalization of \cite[Lemma 4.14]{EqKR}.

\begin{lem} \label{PHD-Uni}
Recall that $\infZ = \mathbb{Z}_+ \cup \{+\infty\}$. Let $K, K' \in \infZ $ and $n \in \Z$. Suppose there are two sequences $\{(s(i),m(i))\}_{i=1}^{K}$ and  $ \{(s'(i),m'(i))\}_{i=1}^{K'}$  in $ \mathbb{Z} \times \infZ $  satisfying:
\begin{itemize}
\item If $i < j$, then $s(i) \le s(j)$ and $s'(i) \le s'(j)$.
\item If $i < j$ and $s(i) = s(j)$, then $m(i) \le m(j)$.
\item If $i < j$ and $s'(i) = s'(j)$, then $m'(i) \le m'(j)$.
\item For every $s \in \Z$, index sets $\{i \ | \ s(i) = s\}$ and $\{i \ | \ s'(i) = s\}$ are finite.
\end{itemize}
If
$$\bigoplus_{i=1}^K H(U_{n,s(i),m(i)}) \cong \bigoplus_{j=1}^{K'} H(U_{n,s'(j),m'(j)})$$ 
as graded $\Fx $-modules, then $K = K'$ and $s(i) = s'(i), m(i) = m'(i)$ for every $1 \le i \le K$.
\end{lem}

\begin{proof}
For a graded $\F $-space $V = \bigoplus_{i \in \Z} V_i$, we consider its graded dimension as the Hilbert series, that is,  $\gdim_\F V := \sum_{i \in \Z} (\dim_\F V_i) t^i$. Denote $$M := \bigoplus_{i=1}^K H(U_{n,s(i),m(i)}) \cong \bigoplus_{j=1}^{K'} H(U_{n,s'(j),m'(j)}).$$ From Corollary \ref{UDec}, we have $H(U_{n,s(i),\infty}) \cong \fxc[n][s(i)] $ and $H(U_{n,s(i),m(i)}) \cong \Fx/(x^{m(i)})||n||\{s(i)\}$ for $m(i) < \infty$. Then graded dimension is well-defined on $M$ since the $\Fx $-grading on $M$ is bounded below and for each fixed $s \in \Z$ the index set is finite.

Let $w_i$ be the generator of $H(U_{n,s(i),m(i)})$ which is homogeneous of degree $s(i)$. For any $l \ge 0$, define $M^{(l)} = x^lM / x^{l+1}M$. Then $M^{(l)}$ is an $\F $-space with homogeneous basis $\{x^lw_i\}_{i \in I_l}$. Here $\{(s(i),m(i))\}_{i \in I_l}$ is a subsequence of $\{(s(i),m(i))\}_{i=1}^{K}$ depending on the choice of $l$, where $m(i) > l$ for all $i \in I_l$ in the subsequence. Consider the graded dimension of $M^{(l)}$. Then $\gdim_\F M^{(l)} =t^l \sum_{i \in I_l} t^{s(i)} = \sum_{i \in I_l} t^{s(i)+l}$. Define $S_{c,l} := \{i \in \Z \ | \ s(i)=c,\ m(i) > l \}$ and similarly $S'_{c,l} := \{i \in \Z \ | \ s'(i)=c,\ m'(i) > l \}$. Then $S_{c,l} \subset I_l$ and the cardinality of $S_{c,l}$ is the coefficient of the $t^{c+l}$ term in $\gdim_\F M^{(l)}$. Therefore each $S_{c,l}$ is finite according to locally finiteness and this coefficient of the $t^{c+l}$ in $\gdim_\F M^{(l)}$ is also the cardinality of $S'_{c,l}$. Thus $S_{c,l}$ and $S'_{c,l}$ have the same (finite) cardinality for each $c$ and $l$. Hence the lemma follows.
\end{proof}

We are now ready to prove Theorem \ref{Decomp}.

\begin{proof}[Proof of Theorem \ref{Decomp}]
The $\Fx $-grading of $PC_n$ is bounded below since the filtration of $C_n$ is bounded below, and $PC$ is locally finite dimensional since $(C,\flt)$ is locally finite dimensional. 
Therefore $(PC,d_x)$ satisfies all the conditions in Proposition \ref{ThmDec}. 
The existence of Decompositions \eqref{eq:DecPC} and \eqref{eq:DecPH} follows from Corollary \ref{UDec}.
Since $PC$ is locally finite dimensional, we know that $\{i \ | \ s(i) = s\}$ is finite for all $s \in \Z$ in the decomposition \eqref{eq:DecPH}. 
Thus, Lemma \ref{PHD-Uni} implies that Decompositions \eqref{eq:DecPC} and \eqref{eq:DecPH} are unique up to chain homotopy and permutation of factors.
\end{proof}

\section{Spectral Sequences}

There are several approaches to compute the spectral sequence of filtered chain complexes. We will use \textit{exact couples} to compute the spectral sequence of summands $\U[m]$ and $\U$. Properties of exact couples can be found in many books such as \cite{Guide2SS} and \cite{Intro2HA}. Let us first briefly review this approach.

\begin{defn}\cite{Guide2SS} \label{def-ex-cp}
An exact couple $\mathcal{C}=(A,E,f,g,h)$ consists of two graded $\F$-linear spaces $A$, $E$ and homogeneous homomorphisms $f$, $g$ and $h$ where the following diagram
\begin{center}
\begin{tikzcd}
A \arrow[rr, "f"] & & A \arrow[ld, "g"] \\
 & E \arrow[lu, "h"] &
\end{tikzcd}
\end{center}
is exact. The differential map $d$ on $E$ is defined to be $d = g \circ h$.

The derived exact couple $\mathcal{C}'=(A,E,f,g,h)'=(A',E',f',g',h')$ is defined by
\begin{itemize}
\item $A' = f(A)$,
\item $E' = \ker d / \im d=H(E,d)$,
\item $f' = f|_{A'}$,
\item $g'(f(a))=g(a)+d(E)$ for any $a \in A$,
\item $h'(b+d(E))=h(b)$ for any $b \in \ker d$.
\end{itemize}

\end{defn}

\begin{defn} \cite{EqKR} \label{def-calH}
For a graded chain complex $K$ of free $\Fx$-modules, the short exact sequence $0 \to K\{1\} \xrightarrow{x} K \xrightarrow{\pi_x} K/xK \to 0$ where $\pi_x$ is the standard quotient map. 
It induces an exact couple $\cH^{(1)}(K)=(A^{(1)},E^{(1)}, f^{(1)}, g^{(1)}, h^{(1)})$ given by 

\begin{center}
\begin{tikzcd}
H(K) \arrow[rr, "x"] & & H(K) \arrow[ld, "(\pi_x)_*"] \\
 & H(K/xK) \arrow[lu, "\Delta"] &
\end{tikzcd}
\end{center}
Denote by $\cH^{(r)}(K)$ the exact couple of $\cH^{(1)}(K)$ derived $(r-1)$ times.

\end{defn}

\begin{lem} \label{SS}
\cite{EqKR}
As graded $\F$-linear spaces,
$$\begin{aligned}
E^{(r)}(\U) & ~~~~\cong~~~~  \fc[n][s]  ~~~~~\forall r \ge 1, \\
E^{(r)}(\U[m]) & ~~~~\cong~~~~  \left\{ \begin{array}{lc}
\fc[n+1][s+m] \oplus \fc[n][s] & 1 \le r \le m, \\
0 & r > m.
\end{array} \right.\\
\end{aligned}$$

In particular, $E^{(\infty)}(\U) \cong E^{(1)}(\U)$ and $E^{(\infty)}(\U[m]) \cong E^{(m+1)}(\U[m])$.

\end{lem}

\begin{proof} 
(Following \cite[Lemma 4.17]{EqKR}.)
First consider $\cH(\U)$. Since the only nonzero term of $\U$ is located at the $n$-th homological degree, the connecting homomorphism $\Delta=h^{(1)}$ is always zero. 
Thus $d^{(r)}=0$ on $E^{(r)}$ for all $r \ge 1$. Therefore $E^{(r)} \cong E^{(1)} \cong \Fx/x\fxc[n][s] \cong \fc[n][s] $ for all $r \ge 1$.

Next consider $\cH(\U[m])$. For $1 \le r < m$, $x^r \cdot \Fx/(x^m) = \mathrm{Span} \{ x^r, x^{r+1}, \ldots , x^{m-1} \} $ is a subspace of $\Fx/(x^m)$. Define the $\F$-linear mapping $x^{-r}$ by
$$x^{-r}:~~~x^r \cdot \Fx/(x^m) \to \Fx/(x^m),~~~~x^{r+i} \mapsto x^i~~~0 \le i \le m-r-1.$$
We will show by induction that for all $1 \le r \le m$:
$$\cH^{(r)}(\U[m]) \cong (x^{r-1} \cdot H(\U[m]), H(\U[m]/x\U[m]), x, \pi_x \circ x^{1-r}, x^{m-1}).$$

For $r=1$, we have 
$$A^{(1)} = H(\U[m]) \cong \Fx /(x^m)||n||\{s\}$$
and
$$E^{(1)} = H(\U[m]/x\U[m]) \cong \fc[n+1][s+m] \oplus \fc[n][s] . $$ 
The nonzero part of $\cH^{(1)}(\U[m])$ is the following part:
\begin{center}
\begin{tikzcd}
\U[m] \{ 1 \}   &  \U[m]   &  \U[m]/x\U[m]  & \\
  & 0 \arrow[r]  & \F\{s+m\}     & {||n+1||}\\
\Fx/(x^m)\{s+1\} \arrow[r,"x"] \arrow[urr, leftarrow, "\Delta"] & \Fx/(x^m)\{s\} \arrow[r,"\pi_x"] & \F\{s\}  & {||n||} \\
0 \arrow[urr, leftarrow] & & & {||n-1||}
\end{tikzcd}
\end{center}
where $\Delta(1)=x^{m-1}$. Thus $\cH^{(1)}(\U[m])=(H(\U[m]),H(\U[m]/x\U[m]),x,\pi_x,\Delta)$.

Assume the inductive hypothesis for $r \le k-1 < m$, and compute $\cH^{(k)}(\U[m])$ by considering the derived exact couple $(x^{k-2} \cdot H(\U[m]), H(\U[m]/x\U[m]), x, \pi_x \circ x^{2-k}, x^{m-1})'$. Then $A^{(k)}=x \cdot A^{(k-1)} \cong x^{k-1} \cdot H(\U[m])$. From the diagram above we may still find that $f^{(k)}=x$ and $h^{(k)}=\Delta=x^{m-1}$. For any $\alpha \in A^{(k-1)}$, we have 
$$g^{(k-1)}(\alpha) = \pi_x \circ x^{2-k}(\alpha) = \pi_x \circ x^{1-k}(x\alpha) = (\pi_x \circ x^{1-k})(f^{(k-1)}(\alpha)).$$
Therefore $g^{(k)}(\alpha) = \pi_x \circ x^{1-k}(\alpha)$. Note that $d^{(k-1)}= g^{(k-1)} \circ h^{(k-1)} = (\pi_x \circ x^{2-k}) \circ x^{m-1}=\pi_x \circ x^{m-k+1}=0$, thus $E^{(k)} \cong E^{(k-1)} \cong H(\U[m]/x\U[m]) \cong \fc[n+1][s+m] \oplus \fc[n][s] $.

Finally consider $r=m+1$. The differential map $d^{(m)}$ becomes $d^{(m)} = g^{(m)} \circ h^{(m)} = (\pi_x \circ x^{1-m}) \circ x^{m-1} = \pi_x \circ 1 = i$ which is an isomorphism on $\fc[n+1][s+m] \oplus \fc[n][s] $, then $E^{(m+1)} = H(E^{(m)},i) \cong 0$. Hence $E^{(r)} \cong E^{(m+1)} \cong 0 $ for all $r > m$. 
\end{proof}

Under our assumptions, the spectral sequence $\{ E^{(r)} \}$ does not necessarily collapse in the usual sense. 
However, from Lemma \ref{SS}, each $E^{(r)}(\U)$ and $E^{(r)}(\U[m])$ both collapse after finite pages. Therefore, we have the following corollary.

\begin{cor} \label{SS-LocCol}
Let $\{E^{(r)}\}$ be the spectral sequence in Theorem \ref{SS-Decomp}.
Then $\{E^{(r)}\}$ collapses locally to $E^{(\infty)}$, which means $\{E_{n,s}^{(r)}\}$ collapses to $E_{n,s}^{(\infty)}$ for every $(n, s) \in \Z \times \Z$.
\end{cor}

We have another observation from Lemma \ref{SS}.

\begin{cor}
Let $K$ be a chain complex of free graded $\mathbb{F}[x]$-module satisfying:
\begin{itemize}
    \item Its differential $d$ is homogeneous and preserves the $\Fx $-grading of $K$;
    \item For each $n$, the $\mathbb{F}[x]$-grading of $K_n$ is bounded below;
    \item For each $n$, $K_n$ is locally finite dimensional over $\F$.
    \end{itemize} 
Then we have that
$$E^{(1)}(K) \cong H(K/xK)$$
and that $\{E^{(r)}(K)\}$ collapses locally to $E^{(\infty)}(K)$ where
$$E^{(\infty)}(K) \cong H(K/(x-1)K).$$
In particular, assume that $K$ is the persistent chain complex of a filtered chain complex $(C,d, \flt)$ over $\F$ satisfying:
\begin{itemize}
    \item[(a)] For each $n$, the filtration $\flt$ on $C_n$ is bounded below;
    \item[(b)] $(C,d,\flt)$ is locally finite dimensional over $\F$.
    \end{itemize}
Then $K/xK$ is isomorphic to the graded chain complex associated to $(C,d, \flt)$, and $K/(x-1)K$ is isomorphic to the filtered chain complex $(C, d, \flt)$.
\end{cor}

\begin{proof}
From Definitions \ref{def-ex-cp} and \ref{def-calH}, we have that $E^{(1)}(K) \cong H(K/xK)$. 
Since $K$ satisfies all conditions in Proposition \ref{ThmDec}, it admits the decomposition in Corollary \ref{UDec}. 
It remains to show that $$E^{(\infty)}(\U) \cong H(\U / (x-1) \U)$$
and 
$$E^{(\infty)}(\U[m]) \cong H(\U[m] / (x-1) \U[m]),$$
where $\U$ and $\U[m]$ are given in Definition \ref{def-decomp-type}.
From ring isomorphism $\Fx / (x-1) \cong \F$, we have
$$\U /(x-1)\U   \cong    0 \to \fc[n][s] \to 0$$
and
$$\U[m] /(x-1)\U[m]  \cong  0 \to \fc[n+1][s+m] \xrightarrow{1} \fc[n][s] \to 0. $$
Hence $H(\U / (x-1) \U) \cong \fc[n][s] $ and $H(\U[m] / (x-1) \U[m]) \cong 0$.
Comparing these to Lemma \ref{SS}, we know that $E^{(\infty)}(K) \cong H(K/(x-1)K)$.

Let $K$ be the persistent chain complex of a filtered chain complex $(C,d, \flt)$ over $\F$ satisfying conditions (a) and (b).
From Definition \ref{def-flt-complex}, recall that $x$ is a homogeneous element of degree $1$ that act on $K$ as the natural inclusion map $i^p : \flt^p C \hookrightarrow \flt^{p+1} C$ for each $p \in\mathbb{Z}$.
Fix $p \in \Z$, we have $x(\flt^{p-1}C) \subset \flt^p C$.
Thus
$$K/xK \cong \bigoplus_{p \in \Z} \flt^p C / \flt^{p-1} C.$$
From Definition \ref{def-pc} and Proposition \ref{PC-iso}, following their notations, we have the differential map $d_x$ of $K$ is given by $(d_x)_n^p=\sum_{q \le p} x^{p-q}d_n^{p,q}$. 
Then $d_x$ induces a differential map $d_{0*}$ in $K/xK$ given by
$$\bigl( d_{0*} \bigr)_n^p = d_n^{p,p} = \pi_{n-1}^{p,p} \circ d_n^p,$$
where $\pi_{n-1}^{p,p} \circ d_n^p : G_n^p \to G_{n-1}^p$ is the restriction of the differential map $d$ on these $\F$-spaces.
Since each $\F$-space $G_n^p$ is fixed by $\flt^p C_n = \flt^{p-1} C_n \oplus G_n^p$, we have $G_n^p \cong \flt^p C_n / \flt^{p-1} C_n$.
Therefore $K/xK$ is isomorphic to the associated graded chain complex of $C$. 

Notice that $x=1$ in the quotient chain complex $K/(x-1)K$.
We know that the image of the homogeneous component of degree $p$ on homological degree $n$ under quotient map is
$$\Bigl( K/(x-1)K \Bigl)_n^p = \bigoplus_{k \ge 0} 1^k G_n^{p-k} = \bigoplus_{q = -\infty}^p G_n^q = \flt^p C_n.$$
Let $d_{1*}$ be the induced differential map of $K/(x-1)K$, then by Definition \ref{def-pc} we have
$$\bigl( d_{1*} \bigr)_n^p = \sum_{q \le p} 1^{p-q} d_n^{p,q} = d_n^p$$
where $d_n^p$ is the differential map $d_n: C_n \to C_{n-1}$ of $(C,d,\flt)$ restricted on $G_n^p$.
Thus $K/(x-1)K$ is isomorphic to the filtered chain complex $(C,d,\flt)$.
\end{proof}

Now we prove Theorem \ref{SS-Decomp}.

\begin{proof}[Proof of Theorem \ref{SS-Decomp}]
The decomposition follows from Theorem \ref{Decomp}. The structures of $E^{(r)}(\U[m]) $ and $E^{(r)}(\U) $ are given in Lemma \ref{SS}. By Corollary \ref{SS-LocCol}, $\{E^{(r)}\}$ collapses locally.
\end{proof}

\section{Relations Between $E^{(r)}$ and $PH$} \label{sec-SS2PH}

Next, we establish the relation between the spectral sequence and the persistent homology. 
From the discussions in the above sections, we know that as $\F $-spaces, the spectral sequence is determined by the persistent chain complex. 
Thus, it is straight forward to give a formula of the $\F $-dimension of the spectral sequence by counting the relevant components in the persistent homology. 
Moreover, we will further show in this section that the persistent homology can also be recovered from the spectral sequence if the filtration is bounded below uniformly.

\begin{cor} \label{Nu=Mul}
Recall that $\mu_n^{i,j}$ is the \textit{persistent multiplicities} from \cite[VII.1]{ComTop} for finite filtered chain complex, and $\nu_{n,s,n}$ be defined in Definition \ref{def-nu}. Then
\begin{itemize}
\item $\mu_n^{i,j}=\nu_{n,i,j-i}$,
\item $\nu_{n,s,m}=\mu_n^{s,s+m}$.
\end{itemize}
Also, the \textit{persistent Betti number} $b_n^{i,j}$ defined in \cite[VII.1]{ComTop} can be expressed in terms of $\nu_{n,s,m}$ as
$$b_n^{i,j}=\sum_{j-m < s \le i} \nu_{n,s,m}.$$

\end{cor}

\begin{proof}

This follows directly from Definition \ref{def-nu} with an index changing. 
The multiplicity $\nu_{n,s,m}$ of $H(\U[m])$ means the number of $n$-homology cycles generating at filtration level $s$ and surviving for $m$ filtration levels. 
By \cite[VII.1]{ComTop}, the persistent multiplicity $\mu_n^{i,j}$ means numbers of $n$-homology cycles generating at filtration level $i$ while getting killed at filtration level $j$. 
Hence $\mu_n^{i,j}=\nu_{n,i,j-i}$ and $\nu_{n,s,m}=\mu_n^{s,s+m}$.
Following the \textit{fundamental lemma of persistent homology} from \cite[VII.1]{ComTop}, we have $b_n^{i,j}= \sum_{s \le i} \sum_{k>j} \mu_n^{s,k}$.
Therefore
$$b_n^{i,j}=\sum_{s \le i} \sum_{k>j} \nu_{n,s,k-s} =\sum_{s \le i} \sum_{m+s>j} \nu_{n,s,m} = \sum_{j-m < s \le i} \nu_{n,s,m}.$$
\end{proof}

Recall that $E_{n,s}^{(r)}$ is the homogeneous component of $E^{(r)}$ with homological grading $n$ and $\Fx$-module grading $s$. Then we have the following proposition.

\begin{prop} \label{E-ns}
For any $r > 0$,
$$\dim_\F E_{n,s}^{(r)} = \nu_{n,s,\infty}+ \sum_{m \ge r} (\nu_{n,s,m} + \nu_{n-1,s-m,m}).$$
\end{prop}

\begin{proof}
From Theorem \ref{SS-Decomp}, the spectral sequence decomposes uniquely as $E^{(r)} \cong \bigoplus_{l=-\infty}^\infty  \bigoplus_{i=1}^{K_l}  E^{(r)}(U_{l,s_l(i),m_l(i)})$. Consider an extended sequence $\{(l,s_l(i),m_l(i))\} \subseteq \Z \times \Z \times \infZ $ and its subsequence $\{(l,s_l(i),m_l(i))\}_{(l,i) \in I_{n,s}}$ where 
$$I_{n,s} :=  \{ (l,i) \in \Z \times \infZ \ | \ 1 \le i \le K_l ,\ \fc[n][s] \subseteq E^{(r)}(U_{l,s_l(i),m_l(i)}) \}.$$
Then $E^{(r)}_{n,s} \cong \bigoplus_{(l,i) \in I_{n,s}} \fc[n][s] $ since each $E^{(r)}(U_{l,s_l(i),m_l(i)})$ contributes at most one summand of $\fc[n][s] $ in $E^{(r)}_{n,s}$ from Lemma \ref{SS}. 
Thus $$\dim_\F E^{(r)}_{n,s} =  \dim_\F \F^{\oplus |I_{n,s}|} = |I_{n,s}|.$$
It suffices to count $|I_{n,s}|$ in terms of $\nu$.

We first separate $I_{n,s} = I_{n,s,\infty} \sqcup I'_{n,s}$ where $I_{n,s,\infty} :=  \{ (n,i) \ | \ 1 \le i \le K_n ,\ s_n(i)=s ,\ m_n(i)=\infty \}$ and $I'_{n,s} :=  I_{n,s} \backslash I_{n,s,\infty}$, because we have $E^{(r)}(\U) \cong  \fc[n][s] $ for all $r \ge 1$ in Lemma \ref{SS}.
Notice that $\nu_{n,s,\infty}= |I_{n,s,\infty}|$.

Next consider $I'_{n,s}$. According to Lemma \ref{SS}, $E^{(r)}(U_{l,s,m}) \cong 0$ for all $m < r$. Also, any rest of the summand $\fc[n][s] $ in $E_{n,s}^{(r)}$ coming from either $E^{(r)}(\U[m]) \cong \fc[n+1][s+m] \oplus \fc[n][s] $ or $E^{(r)}(U_{n-1,s-m,m}) \cong \fc[n][s] \oplus \fc[n-1][s-m] $ for all $m \ge r$.
Therefore $I'_{n,s} = I_{n,s,m \ge r} \sqcup I_{n-1,s-m,m \ge r}$ where 
$$I_{n,s,m \ge r} := \{ (n,i) \ | \ 1 \le i \le K_n ,\ s_n(i)=s ,\ m_n(i) \ge r \} $$
and $$I_{n-1,s-m,m \ge r} := \{ (n-1,i) \ | \ 1 \le i \le K_{n-1} ,\ s_n(i)  =s - m_n(i),\ m_n(i) \ge r \} .$$
Thus $|I_{n,s,m \ge r}| = \sum_{m \ge r} \nu_{n,s,m}$ and $|I_{n-1,s-m,m \ge r}| = \sum_{m \ge r} \nu_{n-1,s-m,m}$.
It is clear that $I_{n,s,m \ge r} \cap I_{n-1,s-m,m \ge r} = \emptyset$ since the index of $I_{n,s,m \ge r}$ coming form homological degree $n$ while the other coming from homological degree $n-1$. 
Hence $$|I_{n,s}| = |I_{n,s,\infty}| + |I_{n,s,m \ge r}| + |I_{n-1,s-m,m \ge r}| = \nu_{n,s,\infty}+ \sum_{m \ge r} (\nu_{n,s,m} + \nu_{n-1,s-m,m}).$$
\end{proof}

Now we prove Theorem \ref{SS=PH}.

\begin{proof}[Proof of Theorem \ref{SS=PH}]
The first part of the theorem follows from Proposition \ref{E-ns}.
If $\{E^{(r)}\}$ is the spectral sequence of the persistent chain complex $(PC,d_x)$ satisfying conditions in Theorem \ref{Decomp}, it suffices to show that
\begin{itemize}
   \item $\nu_{n,s,\infty}=\dim_\F E_{n,s}^{(\infty)}$ and
   \item $\nu_{n,s,m}=\dim_\F E_{n,s}^{(m)} - \dim_\F E_{n,s}^{(m+1)} - \nu_{n-1,s-m,m}$.
  \end{itemize}

From Corollary \ref{SS-LocCol}, $\{E^{(r)}\}$ collapses locally to $E^{(\infty)}$. 
Thus $E_{n,s}^{(\infty)}$ exists for every $n,s \in \Z$ and is also finite $\F$-dimensional. 
It consists of the free part of the homology group with homogeneous grading. 
Suppose $E_{n,s}^{(r)}$ collapses at the $M$-th page, which means $E_{n,s}^{(r)} \cong E_{n,s}^{(r+1)} \cong E_{n,s}^{(\infty)}$ for all $r \ge M$ but $E_{n,s}^{(M-1)} \not \cong E_{n,s}^{(M)}$. 
Consider $\dim_\F E_{n,s}^{(r)} = \nu_{n,s,\infty}+ \sum_{m \ge r} (\nu_{n,s,m} + \nu_{n-1,s-m,m})$ from Proposition \ref{E-ns}. 
For all $r \ge M$, by $E_{n,s}^{(r)} \cong E_{n,s}^{(r+1)}$ we have
$$ \nu_{n,s,r} + \nu_{n-1,s-m,r} =\dim_\F E_{n,s}^{(r)} - \dim_\F E_{n,s}^{(r+1)} =0.$$ 
Thus $\sum_{m \ge M} (\nu_{n,s,m} + \nu_{n-1,s-m,m}) = 0$ and $\dim_\F E_{n,s}^{(\infty)} = \dim_\F E_{n,s}^{(M)} = \nu_{n,s,\infty}$. Hence $\nu_{n,s,\infty}$ is determined by the infinite page $E_{n,s}^{(\infty)}$.

For general $\nu_{n,s,m}$, consider Proposition \ref{E-ns} again. 
We have $$\dim_\F E_{n,s}^{(m)} - \dim_\F E_{n,s}^{(m+1)} = \nu_{n,s,m} + \nu_{n-1,s-m,m}.$$
Thus $\nu_{n,s,m}=\dim_\F E_{n,s}^{(m)} - \dim_\F E_{n,s}^{(m+1)} - \nu_{n-1,s-m,m}$.

If the filtration of $(C,d,\flt)$ is bounded below uniformly, the $\Fx $-grading of $(PC,d_x)$ is therefore bounded below uniformly. 
Then there exists $S \in \Z$, such that $\nu_{n,s,m} = 0$ for all $n \in \Z$, $m \in \infZ $  and $s < S$. 
Thus all $\nu_{n,S,m} = \dim_\F E_{n,S}^{(m)} - \dim_\F E_{n,S}^{(m+1)}$ are in terms of dimensions of spectral sequence.
Finally, each nonzero $\nu_{n,s,m}$ is determined recursively by terms from only the spectral sequence.
\end{proof} 

\bibliographystyle{acm}
\bibliography{PHSS}

\end{document}